\documentclass[11pt,a4paper]{article}
\usepackage[a4paper,margin=1in]{geometry}
\usepackage{amsmath,amsthm,amssymb}
\usepackage{graphicx,color}
\usepackage{enumerate}
\usepackage{tikz}
\usepackage{xspace}
\usepackage{nicefrac}
\usepackage{url}

\definecolor{linkblue}{rgb}{0.1,0.1,0.8}
\usepackage{algorithm,algorithmicx,algpseudocode}

\newcommand{\ignore}[1]{}

\definecolor{darkgreen}{RGB}{0,127,0}

\newtheorem{theorem}{Theorem}[section]
\newtheorem{lemma}[theorem]{Lemma}
\newtheorem{corollary}[theorem]{Corollary}

\newcommand{\eps}{\varepsilon}

\usepackage{authblk}
\usepackage{blindtext}

\newcommand{\soda}[1]{}
\newcommand{\full}[1]{#1}

\title{Push is Fast on Sparse Random Graphs}
\author{Florian Meier}
\author{Ueli Peter}
\affil{ETH Z\"{u}rich}

\date{\today}
\begin{document}
\maketitle 

\begin{abstract}
We consider the classical push broadcast process on a large class of sparse random multigraphs that includes random power law graphs and multigraphs. Our analysis shows that for every $\eps>0$, whp\footnote{\emph{with high probability} here denotes with probability $1-o(1)$.} $O(\log n)$ rounds are sufficient to inform all but an $\eps$-fraction of the vertices. 

It is not hard to see that, e.g. for random power law graphs, the push process needs whp $n^{\Omega(1)}$ rounds to inform all vertices. Fountoulakis, Panagiotou and Sauerwald proved that for random graphs that have power law degree sequences with $\beta>3$, the push-pull protocol needs $\Omega(\log n)$ to inform all but $\eps n$ vertices whp. Our result demonstrates that, for such random graphs, the pull mechanism does not (asymptotically) improve the running time.  This is surprising as it is known that, on random power law graphs with $2<\beta<3$, push-pull is exponentially faster than pull.

\end{abstract}

\thispagestyle{empty}
\clearpage
\setcounter{page}{1}
\section{Introduction}
In the \emph{push} process (or protcol/model) on a graph $G$ with $n$ vertices, initially (in round $0$) only one arbitrary vertex knows the rumor. In each round $i\geq 1$, every vertex that knows the rumor informs uniformly at random one of its neighbors. The push process has a counter part called \emph{pull} process, where in every round, every uninformed vertex asks a random neighbor for the rumor. 

Recently, there has been large interest in analysing push and pull. \full{For general graphs, Giakkoupis \cite{giakkoupis2011tight} continued the work of Chierichetti, Lattanzi and Panconesi \cite{chierichetti2010almost,chierichetti2010rumour} and showed that the \emph{push-pull} protocol (the combination of push and pull) informs \emph{whp} (with high probability, i.e. with probability tending to one as $n$ tends to infinity) every vertex of a graph in $O( \log n/\phi)$ rounds, where $\phi$ denotes the conductance of the graph. }
For the push protocol, Frieze and Grimmet showed in \cite{frieze1985shortest}  that it informs whp all vertices of the complete graph $K_n$ in $(1+o(1))\log_2n+\ln n$. In \cite{fountoulakis2010reliable}, Fountoulakis, Huber and Panagiotou generalised this result to dense random graphs with average degree $\omega(\ln n)$. The first result for push on sparse random graphs was by Fountoulakis and Panagiotou \cite{fountoulakis2012rumor} who proved that, for $d$-regular random graphs and multigraphs, the time needed to spread the rumor is whp $(1+o(1))c_d\ln n$, where $c_d$ is a constant that depends on $d$. 

A class of sparse random graphs that has been studied in many contexts is the class of random graphs with a degree sequence that follows a \emph{power law}, i.e. the number of vertices of degree $k$ is proportional to $k^{-\beta}$
. Power law graphs have gained considerable attention in recent years as the empirical evidence suggests that many of today's largest real world networks, both physical, such as the internet, and virtual, like social networks or the web graph, follow such a degree sequence \cite{mitzenmacher2004brief}. It is therefore of particular importance that the building blocks of distributed systems, such as broadcasting protocols, perform well on graphs with power law degree sequence. 

Random graphs with power law degree sequences contain whp many vertices of small degree that are only connected to vertices of very large degree (i.e. of degree $n^{\Omega(1)}$). It follows that the push protocol needs whp $n^{\Omega(1)}$ rounds to inform all these vertices.
For the push-pull protocol, Fountoulakis, Panagiotou and Sauerwald \cite{fountoulakis2012ultra} showed that the time needed to spread the information to at least an $(1-\eps)$-fraction of the vertices of a random power law graph\footnote{They use the Chung-Lu model \cite{chung2004average}, which is slightly different than the uniform random graph model considered in this paper.} with $2<\beta<3$ is whp of order $O(\log \log n)$ while if $\beta>3$ it is whp $\Omega(\log n)$. We follow their line of research and show that for power law random multigraphs the push model suffices to inform whp a $(1-\eps)$-fraction of the vertices (for every $\eps>0$) in logarithmic time. Our result holds for a much larger class of sparse random multigraphs, which contains random regular graphs, random bounded degree graphs and many more and generalises for most of those classes (unfortunately not for power law random graphs with $\beta<3$) to the corresponding class of random (simple) graphs. Note that for push the number of informed vertices can at most double in every round. Thus, $\Omega(\log n)$ is a lower bound for every graph, and our result therefore implies that, while for power law networks with $2<\beta< 3$ push-pull is exponentially faster than push (by \cite{fountoulakis2012ultra}), for $\beta>3$, the pull mechanism does not lead to an (asymptotical) improvement. 

\subsection{Our Contribution}
\label{sec:our_contribution}
For a graph or multigraph $M=(V, E)$, let $T(M)$ denote the time until all vertices are informed by the push protocol, and for $0 \le \eps\le 1$ let $T_{\eps}(M)$ denote the time until all but at most $\eps |V|$ vertices are informed. Before we state the theorem, we quickly introduce the random graph model. 

A \emph{degree sequence} on $n$ vertices is a sequence of integers $D_n= d_{1,n}, d_{2,n}, \dots , d_{n,n}$, where $d_{i,n}$ is the degree of the $i$-th vertex for $1\le i\le n$. We consider infinite sequences $\mathcal{D}=D_1, D_2, \dots$, such that for each $i\geq 1$ $D_i$ is a degree sequence of length $i$. We say that such a sequence $\mathcal{D}$ is an \emph{asymptotic degree sequence} if there exist a sequence of real numbers $\lambda_1, \lambda_2, \dots$ such that  
$$
\lim_{n\rightarrow\infty} \frac{|\left \{j\middle| d_{j,n}= i\right\}|}{n}=\lambda_i 
\quad \text{ for every $i\geq 1$ and } 
 \quad \lim_{n\rightarrow\infty}\sum_{j=1}^nd_{j,n}=n\sum_{j=0}^{\infty}j\lambda_j.$$
Moreover, we say that $\mathcal{D}$ is  \emph{2-smooth} if 
$$
\sum_{j=1}^{\infty}j^2\lambda_j<\infty.
$$
A sufficiently large (multi)graph created from a 2-smooth asymptotic degree sequence is always \emph{sparse} in the sense that the number of edges is linear in the number of vertices, since
\begin{equation}
\label{eq:two_smooth_sparse}
\lim_{n\rightarrow\infty} \frac{\sum_{j=1}^nd_{j,n}}{n}= \sum_{j=1}^{\infty}j\lambda_j\le  \sum_{j=1}^{\infty}j^2\lambda_j<\infty.
\end{equation}
Conversely, not every sparse degree sequence is $2$-smooth (we demonstrate this later on the example of a power-law with $\beta<3$).

We denote by 
$$\Delta:=\Delta(\mathcal{D},n)=\max\left\{d_{i,n}\mid 1\le i\le n\right\}$$
 the maximum degree of an asymptotic degree sequence $\mathcal{D}$.
Moreover, let $\delta(\mathcal{D}):=d-1$ if there exist an integer $d$ such that $\lambda_d=1$. Otherwise, let $\delta$ be the minimum degree. By this definition of $\delta$ we assure that there are linearly many vertices of degree larger than $\delta$.

Let $M(\mathcal{D})$ be the sequence of random multigraphs such that the $i$-th element $M_i(\mathcal{D})$ is chosen uniformly at random from all multigraphs with degree sequence $D_i$. We say that $M(\mathcal{D})$ has a given property w.h.p if the probability that $M_n(\mathcal{D})$ has this property goes to $1$ as $n$ tends to infinity. We use the same notation for simple graphs but we write $G(\mathcal{D})$ and $G_i(\mathcal{D})$
 instead of $M(\mathcal{D})$ and $M_i(\mathcal{D})$.
Similar definitions of fixed degree sequences for random graphs have been used in \cite{molloy1998size}, \cite{fernholz2004cores} and \cite{fernholz2007diameter}.

\begin{theorem}[Main Result]
\label{thm:main_theorem}
Let $\mathcal{D}$ be a sparse asymptotic degree sequence with $\delta\geq3$, maximum degree $\Delta$ and let $c_{\mathcal{D}}:=\frac{1}{\ln\left(2\left(1-\frac{1}{\delta} \right)\right)}$. 

If $\Delta=o(\sqrt{n})$, then for every $n$, every $\eps>0$ and every $c>c_{\mathcal{D}}$ we have w.h.p
$$T_{\eps}(M(\mathcal{D})) \le c \ln n.$$

If $\Delta$ does not depend on $n$, then we have for every $n$ with probability $1-o(1/n)$
$$T(M(\mathcal{D})) =O(\ln n).$$
\end{theorem}
Let us first remark to what extent our theorem extends to simple graphs. 
It is well known (see Section~\ref{sec:degree_sequences}) that for a 2-smooth degree sequence $D_n$ with maximum degree $o(\sqrt{n})$ a non negligible fraction of all multigraphs is simple, and therefore that every graph property that holds w.h.p in $M(D_n)$ also holds w.h.p in $G(D_n)$.  We restate the first statement of our main theorem for simple graphs in the following corollary. The second statement is already known for simple graphs. Since random bounded degree graphs have w.h.p a diameter of order $O(\ln n)$ (see \cite{fernholz2007diameter}), it follows from a result by Feige et. al. \cite{feige1990randomized} which states that $T(G)=O(\ln n)$ for every bounded degree simple graph $G$ with diameter at most $O(\ln n)$.
 
\begin{corollary}
\label{cor:main_cor}
Let $\mathcal{D}$ be a $2$-smooth asymptotic degree sequence with $\delta\geq3$, maximum degree $\Delta$ and let $c_{\mathcal{D}}:=\frac{1}{\ln\left(2\left(1-\frac{1}{\delta} \right)\right)}$. 

If $\Delta=o(\sqrt{n})$, then for every $n$, every $\eps>0$ and every $c>c_{\mathcal{D}}$ we have w.h.p
$$T_{\eps}(G(\mathcal{D})) \le c \ln n.$$
\end{corollary}

Random regular graphs, random bounded degree graphs and random power-law graphs with $\beta>3$ all have $2$-smooth degree sequences. For a power-law with $\beta<3$, we have
$$\sum_{i=1}^{\infty} i^2 \lambda_i=\sum_{i=1}^{\infty} \Theta(i^{2-\beta})= \omega(1),$$
and the degree sequence is therefore not $2$-smooth. Unfortunately, this means that Corollary~\ref{cor:main_cor} does not include power-law degree sequences with $2<\beta<3$. But, since on power-law random graphs with $2<\beta<3$ push-pull is (at least) exponentially faster than push, the simple push protocol is more interesting for $\beta>3$. Moreover, since multiple edges and self-loops seem to be a disadvantage for fast rumor spreading, we believe that push is also fast for simple random power-law graphs in that regime but our methods can not be used to prove it. 

Requiring $\delta\geq 3$ is not a strong limitation, since $\lambda_2=0$ and minimum degree two is necessary to ensure that the random graph is w.h.p connected (see \cite{luczak1989sparse} for more details). The only sparse degree sequences that generate w.h.p connected graphs, but are not covered by our theorem, are $3$-regular and almost $3$-regular graphs (where only $o(n)$ vertices do not have degree $3$). Our proof technique would allow us to handle those graphs, but we omit the necessary case distinction for the sake of simplicity. For $d$-regular random graphs, Fountoulakis and Panagiotou proved in \cite{fountoulakis2012rumor} that w.h.p $T(G(\mathcal{D}))=(1+o(1))c_d\ln n$, where 
$$c_d:=\frac{1}{\ln\left(2\left(1-\frac{1}{d} \right)\right)} - \frac{1}{d\ln\left(1-\frac{1}{d} \right)}.$$ 
Note that $c_d>c_{\mathcal{D}}$ holds for $d$-regular degree sequences $\mathcal{D}$ and $d\geq 4$. Therefore, even for very small $\eps>0$, the time needed to inform the last $\eps n$ vertices is not negligible. This might come as a surprise first, but it is rather simple to explain. In a $d$-regular graph, the probability that a vertex does not receive the rumor, even if all its neighbors are informed, is $\left(\frac{d-1}{d}\right)^d$ 
in every round. Thus, even for $t\approx -(\frac{1}{d\ln(1-1/d)})\ln n$ rounds, the probability that a vertex does not get the information in $t$ rounds, is $\Omega(1/n)$. For a linear fraction of non-informed vertices, we expect that constantly many remain non-informed after $t$ rounds. For large $d$ the difference between $c_d$ and $c_{\mathcal{D}}$ tends to exactly this $-\frac{1}{d\ln(1-1/d)}$. For small $d$ there is a gap between $c_d-\frac{1}{d\ln(1-1/d)}$ and $c_{\mathcal{D}}$, but by treating random regular graphs separately, our proof could be optimized to yield the correct constant.

\full{Having answered these questions about push, it is natural to ask for the running time of its symmetric counter part pull. Let $\bar{T}(G)$ denote the time pull needs to inform all vertices of a multigraph $M$ and similarly, let $\bar{T}_{\varepsilon}(M)$ denote the time it needs to inform all but $\varepsilon n$ vertices for a fixed $\varepsilon>0$. It is well known (see for example \cite{giakkoupis2011tight}) that for every multigraph $M$, every upper bound that holds with probability $1-f(n)$ for $T(M)$, holds with probability $1-O(n\cdot f(n))$ for $\bar{T}(M)$ and vice versa. For bounded degree sequences, our result therefore implies the second statement of the Corollary~\ref{cor:main} below. Note that this does not follow from \cite{fountoulakis2012rumor}, since the result there holds only with probability $1-o(1)$.

For power-law degree sequences, we observe that push and pull behave very differently. While push fails in informing the last few vertices efficiently, depending on the choice of the initial vertex, pull might need polynomially many rounds to inform the second vertex. More precisely, in a power-law random graph, there is w.h.p a vertex $v$ of degree $O(1)$ that has only neighbors of degree at least $n^{\Omega(1)}$. If $v$ is the initial vertex, then the expected time until the second vertex gets informed is $n^{\Omega(1)}$. This observation implies that for a power-law random multigraph $M$, w.h.p $\bar{T}_{\varepsilon}(M)=n^{\Omega(1)}$ for some initial vertices. However, Theorem~\ref{thm:main_theorem} implies that $\bar{T}_{\varepsilon}(M)=O(\log n)$ holds for almost all initial vertices of a sparse random multigraph $G$. 

\begin{corollary}\label{cor:main}
Let $\mathcal{D}$ and $c_{\mathcal{D}}$ be as defined in Theorem~\ref{thm:main_theorem}.
If $\Delta=o(\sqrt{n})$, then w.h.p for every $n$, every $\eps>0$ and every $c>c_{\mathcal{D}}$ we have w.h.p 
$$\bar{T}_{\eps}(M(\mathcal{D})) \le c \ln n$$
for all except $o(n)$ choices of the initial vertex.
If $\Delta$ does not depend on $n$, then for every $n$ we have w.h.p
$$\bar{T}(M(\mathcal{D})) = O(\ln n).$$
\end{corollary}

An interesting open question is whether $\bar{T}(M(\mathcal{D}))=O(\log n)$ holds w.h.p for almost all initial vertices of a random multigraph $M(\mathcal{D})$ with a sparse degree sequence $\mathcal{D}$. We conjecture that this is indeed the case. 

}

\section{Properties of the Degree Sequence and the Configuration Model}
\label{sec:degree_sequences}
In the previous section, we introduced the concept of a sparse asymptotic degree sequence. Let us here comment on two properties of such a degree sequence. 
Recall that we require the existence of an infinite sequence of real numbers $\lambda_1, \lambda_2, \dots$ such that
\begin{equation}
\label{eq:cumulate}
\lim_{n\rightarrow\infty} \frac{|\left \{j\middle| d_{j,n}= i\right\}|}{n}=\lambda_i.
\end{equation}
Thus, it follows that $\lim_{n\rightarrow\infty}\sum_{i=0}^n\lambda_i=1$.
Intuitively, $\lambda_i$ denotes the fraction of vertices of degree $i$ and we have $(\lambda_i\pm o(1))n$ vertices of degree $i$ in $M(D_n)$. 
Moreover, note that sparseness implies that  
\begin{equation}
\label{eq:only_few_large}
\lim_{i\rightarrow \infty }\lambda_i i=0.
\end{equation}
This observation allows us to handle sparse graphs almost like bounded degree graphs.  

A common method for generating random multigraphs with a fixed degree sequence $D_n$ is to create for every $1\le i\le n$ a vertex $v_i$ with $d_{i,n}$ \emph{stubs} (one half of an edge) attached to it, and connect the stubs by choosing a random \emph{configuration}, i.e. a random matching of the stubs. This random graph model is well known as the \emph{configuration model}, and it was introduced by Bollob{\'a}s \cite{bolloconf} and Bender and Canfield \cite{bender1978asymptotic} independently.  

Note that instead of choosing a configuration uniformly at random, we can iteratively choose an arbitrary unmatched stub and \emph{match} it to a stub that we choose u.a.r. from all unmatched stubs. This technique is called the \emph{principle of deferred decisions} (see \cite{mitzenmacher2005probability}), i.e. we delay every random choice until it is most convenient for our analysis. For the remainder of this paper, we always construct the random matching using deferred decisions. Whenever we match a stub, we denote by $U$ the set of all unmatched stubs and by $e\in_r U$ a stub $e$ that we choose u.a.r. from $U$.

It is a weakness of the configuration model that it creates multigraphs, i.e. the created graph may contain multiple edges and loops. In \cite{janson2009probability} Janson proved that for a degree sequence $D_n=d_{1, n}, \dots, d_{n, n}$ with maximum degree $o(\sqrt{n})$, the probability that a random multigraph $M(D_n)$ is simple is
$$\Pr[M(D_n)\text{ is simple}]=e^{-\frac{1}{16 |E|^2}\left(\sum_{i}d_i^2 \right)^2+\frac{1}{4}} +o(1).$$ 
This implies that for a 2-smooth degree sequence $D_n$ with maximum degree $o(\sqrt{n})$ a non negligible fraction of all configurations is simple, and therefore that every graph property that holds w.h.p in $M(D_n)$ also holds w.h.p in $G(D_n)$. \full{For a sparse degree sequence $D_n$ that is not 2-smooth, an event that holds with probability at least $1-o(\Pr[M(D_n)\text{ is simple}])$ in $M(D_n)$ holds w.h.p in $G(D_n)$. Thus, for a power-law degree sequence $D_n$ with $2<\beta<3$, we would have to prove that our statement holds with probability at least $1-o(e^{-n^{3/\beta -1}})$ for $M(D_n)$, to prove that it holds w.h.p for $G(D_n)$.}

\section{Tools}
In our proofs we apply at several occasions the well known Chernoff-Hoeffding bound (see e.g. Chapter~1 of \cite{dubhashi2009concentration}) for sums of independent Bernoulli random variables. Sometimes we need those bounds in the following slightly dependent setting. The proof of the following theorem follows directly from the Chernoff bound and is left as an exercise for the reader (see for example Problem~1.7 in \cite{dubhashi2009concentration}).

\begin{theorem}
\label{thm:chernoff}
Let $p, q\in [0,1]$ and let $X_1, \dots, X_n\in \{0,1\}$ be $n$ indicator variables and $X:=\sum_{i=1}^nX_i$. If for each $1\le i \le n$
$$\mathbb{E}[X_i|X_1, \dots, X_{i-1}] \geq p \quad \text{ and } \quad \mathbb{E}[X_i|X_1, \dots, X_{i-1}] \le q, $$
then it holds for every $0<\eps<1$ that 
$$\Pr[X\geq (1+\eps)nq]\le e^{-nq\eps^2/3} \quad \text{ and } \quad \Pr[X\le (1-\eps)np]\le e^{-np\eps^2/2}.$$
\end{theorem}

We also need the following form of another classical type of concentration inequality, often referred to as the method of bounded differences or as Azuma's inequality (Theorem~5.3 in \cite{dubhashi2009concentration}).  
\begin{theorem}[Method of Bounded Differences]
\label{thm:bounded_differences}
Let $d_1\dots, d_n$ be a sequence of constants and let $f:=f (x_1 , \ldots , x_n )$ be a function that satisfies for all $1\le i\le n$ 
\begin{align*}
|f(a)-f(a')| &\leq d_i
\end{align*}
whenever $a$ and $a'$ differ in just the $i$-th coordinate. If $X_1 , \ldots , X_n$ are independent random variables, then
$$
\Pr[f > \mathbb{E}[f]+t]\leq e^{-\frac{t^2}{2\sum_{i=1}^n d_i^2}} \quad \text{ and } \quad  \Pr[f < \mathbb{E}[f]-t]\leq e^{-\frac{t^2}{2\sum_{i=1}^n d_i^2}}.
$$
\end{theorem}

\section{Proof of the Main Results}
We use the configuration model to prove our main theorem for $M(\mathcal{D})$. Recall that we can match the stubs in the configuration model in an arbitrary order. For our purpose, it is very convenient to match a stub only when it is selected by the push process. Thereby, we combine the process that spreads the rumor with the process that matches the stubs. 
An important element of our proof is that we sometimes \emph{delay} a push. This means that for an unmatched stub $e$ that is selected by the push protocol, we postpone matching $e$ and thus also the push. A delayed stub may be matched and pushed in a later round or it may as well be omitted till the end. Note that delaying pushes can not make the spreading faster and therefore, since we are only interested in upper bounds on the rumor spreading, it is legitimate to delay pushes. We denote the process that combines the delayed push process with the matching of the stubs as the \emph{delayed random graph push process} or \emph{DRP process}.
For the proof of Theorem~\ref{thm:main_theorem}, we split the execution of the DRP process in three phases. 


Before we explain the three phases and state the main lemmas, we need to introduce some additional notation and constants.  
Let $0 < \gamma \le 1/6$ be a small enough constant, such that 
\begin{equation}
\label{eq:cdondition_gamma}
\frac{1-\lambda_{\delta}}{64\delta^2} >\frac{2\gamma}{1-\gamma} \quad \text{ and } \quad \gamma<\frac{1}{\delta},
\end{equation}
which is possible since $\lambda_{\delta}<1$.
Moreover, let $M$ be the smallest integer such that $\frac{\sum_{j=\delta}^M\lambda_j j}{\sum_{j=\delta}^{\infty}\lambda_j j}\geq 1-\frac{\gamma}{4}$ and let

\begin{equation}
\label{eq:cdondition_alpha}
\alpha :=\min\left\{\frac{\gamma\sum_{j=\delta}^{\infty}\lambda_j j}{2\cdot M }, \frac{1-\lambda_{\delta}}{1+\lambda_{\delta}}, \frac{(\delta-1)}{4\delta(1-\frac{1}{\delta})}\right\} \stackrel{\eqref{eq:only_few_large}}{=} \Theta(1).
\end{equation}
For $i\geq 1$ let $I_i$ denote the set of \emph{informed} vertices after round $i$ and let $N_i$ denote the set of \emph{newly informed} vertices in round $i$. Note that this definition of $\alpha, M$ and $\gamma$ assures that for $|I_i|< \alpha n$ the probability to match to an uninformed vertex of degree at most $M$, when choosing u.a.r. from the unmatched stubs $U$ in round $i$, is at least
\begin{equation}
\label{eq:prob_not_informed_small}
(1-o(1))\frac{n\sum_{j=\delta}^M\lambda_j j - |I_i|M}{n\sum_{j=\delta}^{\infty}\lambda_j j} \geq (1-o(1))(1-\frac{3\gamma}{4})>1-\gamma.
\end{equation}
We neglect vertices of degree larger than $M$ whenever possible. Let therefore $\bar{N}_i$ denote the set of non sleeping vertices of degree at most $M$ that have been newly informed in round $i$.

With this definitions at hand, we are ready to explain the three phases. For the sake of simplicity, we do not repeat the whole framework in the statements of the following lemmas, but we always make the assumption of Theorem~\ref{thm:main_theorem}. 

In the beginning, it is very unlikely that a vertex is informed more than once. Therefore, the rumor spreading builds a tree with branching factor at least $\delta-1$. This guarantees exponential growth as $\delta-1\geq 2$. \soda{The formal statement of this observation (Lemma~\ref{lem:phase_1}) follows by basic calculation and we therefore omit the proof in this conference paper.} 
\begin{lemma}
\label{lem:phase_1}
Let $t_1$ be the smallest integer such that $|\bar{N}_{t_1}|\geq \log^5 n$. It holds w.h.p that we can delay the DRP process so that $t_1=O((\log\log n)^2)$ and $|I_{t_1}|=O(\log^5 n)$. Moreover, if the the maximum degree $\Delta$ is constant, we can delay the DRP process so that $t_1=O(\log n)$ holds with probability $1-o(1/n)$. 
\end{lemma}

The second phase is longer and its running time dominates the total running time of the rumor spreading. In this phase, we want to inform a linear fraction of the vertices, and we therefore have to deal with vertices that are informed more than once. \full{But since we have many vertices of degree larger than $\delta$, we can compensate for this and show that there is still exponential growth. }The proof of the following lemma contains the main ideas of our analysis. 
\begin{lemma}
\label{lem:phase_two}
Let $t_2$ be the smallest integer such that $|I_{t_2}|\geq \alpha n$. It holds with probability $1-o(1/n)$ that we can delay the DRP process so that $t_2-t_1\le c \ln n$. 
\end{lemma}
In the third phase, we handle the spreading to the remaining vertices for random graphs of bounded and unbounded maximum degree. 
\begin{lemma}
\label{lem:phase_3}
For $\eps>0$, let $t_3$ be the first round such that $|I_{t_3}|>(1-\eps)n$ and let $t_4$ be the first round such that $|I_{t_4}|=n$. It holds w.h.p that we can delay the DRP process so that $t_3-t_2 =  o(\log n).$ Moreover, if $\Delta$ does not depend on $n$, then it holds with probability $(1-o(1/n))$ that we can delay the DRP process so that $t_4-t_2 =O(\log n).$
\end{lemma}
The three lemmas together clearly imply Theorem~\ref{thm:main_theorem}.\full{ We prove Lemma~\ref{lem:phase_1} in Section~\ref{sec:phase2}, Lemma~\ref{lem:phase_two} in Section~\ref{sec:phase2} and Lemma~\ref{lem:phase_3} in Section~\ref{sec:phase3}.}\soda{ In this conference paper, we focus on the proof of Lemma~\ref{lem:phase_two} (in Section~\ref{sec:phase2}). 
At the time this paper is published, a full version that contains all the proofs will be available on the second authors web site.} We believe that the presented proofs give a  good intuition on the dynamics of the rumor spreading process, and that the techniques we use could be helpful for the analysis of other processes on random multigraphs. 

\full{\subsection{Phase 1}\label{sec:phase1}
This section is devoted to the proof of Lemma~\ref{lem:phase_1}.
\begin{proof}[Proof of Lemma~\ref{lem:phase_1}]
We consider the process until for the first time at least $\log^5n$ vertices become newly informed. We control the process (by delaying pushes) so that the informed vertices build a 3-regular tree. Let $T_1:=\{v\}$, where $v$ is the initial vertex and for $j>1$, let $T_j$ denote the set of vertices on the $j$-th level of the tree. On level $j\geq1$, we delay all pushes for $t_j$ (to be defined later) rounds. For every vertex $v\in T_j$, let $E_v$ denote the set of unmatched stubs that have been selected by the push process during those $t_j$ rounds. We select from every set $E_{v}$ at most two (three on the root) stubs and match them u.a.r. to stubs in $U$. Finally, we inform the first $\lceil(4/3)^{j-1}\rceil$ (or all if there are less) vertices that have been matched to vertices in $T_{j}$, and denote the set of newly informed vertices by $T_{j+1}$. We iterate this process as long as $j\le h$ for
$$h:=\frac{5\log\log n}{\log (4/3)}+2.$$
Clearly, the time we need to build this tree is $\sum_{i=1}^ht_j$, and we claim that we can choose the $t_j$'s in a way that the statement of the lemma is satisfied. 

We say that a branch of the tree \emph{dies} if the corresponding stub matches to a vertex that is already informed or if a vertex does not select at least two new stubs in $t_j$ rounds. For level $j$, we observe that there are $|U|\geq (n-\sum_{i=1}^{j-1}|T_i|)3$ unmatched stubs and at most $(\sum_{i=1}^{j-1}|T_i|)\Delta$ of those stubs are already in the tree. Thus, the probability to match to a stub in the tree is at most
$$\frac{(\sum_{i=1}^{j-1}|T_i|)\Delta}{(n-\sum_{i=1}^{j-1}|T_i|)3}.$$
Moreover, the probability that a branch dies because some other stub on the same level matches to it is at most
$$\frac{2|T_j|}{(n-\sum_{i=1}^{j-1}|T_i|)3},$$
and the probability that a branch dies because the corresponding vertex did not select two stubs during the $t_j$ rounds is at most $\left(2/3\right)^{t_j}$.
Altogether, the probability that a branch at level $j$ dies is at most
$$\chi_j:= \frac{(\sum_{i=1}^{j-1}|T_i|)\Delta+ 2|T_j|}{(n-\sum_{i=1}^{j-1}|T_i|)3}+\left(\frac{2}{3}\right)^{t_j}.$$
Let 
$$p_j:= \Pr\left[|T_j|<\left(\frac{4}{3}\right)^{j-1}\middle| |T_{j-1}|\geq \left(\frac{4}{3}\right)^{j-2}\right],$$
and note that under the assumption of $|T_{j-1}|\geq \left(\frac{4}{3}\right)^{j-2}$, at least $\frac{2}{3}(\frac{4}{3})^{j-2}$ of the $2(\frac{4}{3})^{j-2}$ branches on level $j$ have to die, to satisfy the event $|T_j|<\left(\frac{4}{3}\right)^{j-1}$. The probability for this is at most
\begin{equation}\label{eq:p_j_first_lemma}
p_j\le \binom{2(\frac{4}{3})^{j-2}}{\frac{2}{3}(\frac{4}{3})^{j-2}}\cdot \chi_j^{\frac{2}{3}(\frac{4}{3})^{j-2}}\le 2^{2(\frac{4}{3})^{j-2}} \chi_j^{\frac{2}{3}(\frac{4}{3})^{j-2}}=\left(4\chi_j^{\frac{2}{3}}\right)^{(\frac{4}{3})^{j-2}}.
\end{equation}
Moreover, it is not very hard to see that $|T_3|<(4/3)^2$ holds only if at least three branches on the first two levels die. The probability for this is at most
\begin{equation}\label{eq:first_lem_initial_prob}
\Pr[|T_3|<(4/3)^2] =O(\chi_3^3).
\end{equation}
Combining \eqref{eq:p_j_first_lemma} and \eqref{eq:first_lem_initial_prob}, we conclude that the probability for $|T_h|< (\frac{4}{3})^{h-1}=\frac{4}{3}\log^5 n$ is at most
\begin{equation}
\label{eq:tree_conditional_sum}
\Pr[|T_3|<(4/3)^2] +\sum_{j=4}^h p_j \le O(\chi_3^3)+\sum_{j=4}^h\left(4\chi_j^{\frac{2}{3}}\right)^{(\frac{4}{3})^{j-2}}.
\end{equation}

With this at hand, we can prove the first statement of the lemma. Let $t_j=\frac{3\log\log n}{2\log(3/2)}$ for all $1\le j\le h$. Clearly, the number of rounds we need to build $T_h$ is 
$$\sum_{j=1}^ht_j=O((\log\log n)^2), $$ 
and, since
$$\chi_j= O\left(\frac{2^i}{\sqrt{n}}+\frac{1}{\log^{3/2}n}\right)=O\left(\frac{1}{\log^{3/2}n}\right),$$
 the probability that there are less than $\frac{4}{3}\log^5 n$ newly informed vertices in $T_h$ is therefore by \eqref{eq:tree_conditional_sum} at most $O(\log\log n/\log^{4/3} n)=o(1)$. It remains to show that at least $\log^5n$ of these vertices have degree at most $M$. It follows from \eqref{eq:prob_not_informed_small} that every newly added vertex is with probability at least $1-\gamma$ of degree at most $M$. Moreover, this probability bound holds independently of the degrees of all the other vertices in the tree. Thus, it follows from Theorem~\ref{thm:chernoff} that with probability $1-n^{-\omega(1)}$ at least $\log^5 n$ of the newly informed vertices are of degree at most $M$. 
 
 For bounded degree graphs, we set $t_j:=36\log n/(j^2\log (3/2))$ for $1\le j\le h$, and the time needed to build the tree is therefore at most
 $$\frac{36\log n}{\log(3/2)}\sum_{j=1}^h\frac{1}{j^2}=O(\log n).$$
 Moreover, since for $1\le j \le h$ we have that 
 $$\chi_j =O\left(\frac{2^j}{\sqrt{n}}\right)+O\left(n^{-\frac{36}{j^2}}\right), $$
 it is not hard to see that the probability that there are less than $\frac{4}{3}\log^5 n$ newly informed vertices in $T_h$ is by \eqref{eq:tree_conditional_sum} at most
 $$O\left(\chi_3^3\right)+\sum_{j=4}^h\left(4\chi_j^{2/3}\right)^{(4/3)^{j-2}}=o(1/n).$$
 Finally, since we only allow $\lceil(\frac{4}{3})^{j-1}\rceil$ vertices on level $j$ for $1\le j\le h$, the number of informed vertices is at most $O(\log^5n)$.
\end{proof}}

\subsection{Phase 2}\label{sec:phase2}
Lemma~\ref{lem:phase_two} states that the exponential growth, exploited in the first phase, continues until a linear fraction of the vertices is informed. 
Unfortunately, when $|I_i|$ is large, we can not neglect the probability that a selected stub is matched to an already informed vertex. We call this a \emph{back match}. 
On the other hand, since we have many vertices of degree strictly larger than $\delta\geq 3$, the branching factor is often larger than $\delta-1\geq 2$. 

First, we introduce a $(\delta-1)$-ary random tree process. It is not very hard to show that this tree process has exponential growth. Then, we couple the DRP process to the tree process. Our coupling guarantees that the number of newly informed vertices of degree at most $M$ is always exactly the number of new born vertices in the tree process, and we thereby prove an exponential lower bound on the growth of the DRP process. 

\subsubsection{The Tree Process}\label{sec:tree_process} The tree process with branching factor $\delta-1$ is defined as follows. 
Every vertex has $\delta$ stubs attached to it. A stub is either \emph{free}, if it has never been selected, or \emph{used}, if it is already matched. We start in round $t_1+1$ with $\lceil\log^5 n\rceil$ vertices, of which each one has exactly one stub that is used and the other $\delta-1$ stubs are free. Note that the time shift of $t_1$ is convenient to couple the tree process with phase 2 of the push process. In round $i>t_1$, we select at every vertex one stub u.a.r from all its stubs. We denote the set of free stubs that have been selected in round $i$  as $S_i^T$. Algorithm~\ref{alg:tree_algorithm} then iterates trough all stubs $e\in S_i^T$ and creates a new vertex $u$ at $e$.
\begin{algorithm}
\caption{Tree process for $S_i^T$}
\label{alg:tree_algorithm}
\begin{algorithmic}
\For{$e\in S_i^T$}
	\State Create a new vertex $u$ and match a stub of $u$ to $e$.
\EndFor
\end{algorithmic}
\end{algorithm}

For $i\geq t_1+1$, let $N^T_i$ be the set of newly created vertices in round $i$, and let $P^T_i$ be the set of free stubs after round $i$. The following lemma quantifies the exponential growth.
\begin{lemma}
\label{lem:tree_model}
Let $|N^T_{t_1}|= \lceil\log^5 n\rceil$ and let $\delta\geq 3$. It holds with probability $1-o(1/n)$ for all $t_1+1\le i =O(\log n)$ that 
\begin{equation}
\label{eq:stub_change}
 |P^T_i|  = (1\pm o(1)) |P^T_{t_1}| \left(2\left(1-\frac{1}{\delta}\right) \right)^{i-t_1}
\end{equation}
and 
\begin{equation}
\label{eq:newly_informed_change}
 |N^T_i|=(1\pm o(1))\frac{1}{\delta}|P_{i-1}^T|.
\end{equation}
\end{lemma}
\soda{In this conference paper, we omit the proof of this rather technical lemma.}
\full{
\begin{proof}
In every round $i$, a free stub $e\in P_{i-1}^T$ creates $\delta-1$ new stubs with probability $1/\delta$, or remains free with probability $(\delta-1)/\delta$. This observation already implies that the expectation of $|P_i^T|$ conditioned on $|P_{i-1}^T|$ is
$$\mathbb{E}\left[|P_i^T|\middle| |P_{i-1}^T|\right]  =  |P^T_{i-1}| \left(\frac{2(\delta-1)}{\delta} \right)$$
and that the expected number of new vertices in round $i$ given the number of free stubs in round $i-1$ is 
 $$\mathbb{E}\left[|N_i^T|\middle| |P_{i-1}^T|\right]  =  |P^T_{i-1}| \frac{1}{\delta}.$$
We start with $|P^T_{t_1}|\geq (\delta-1)|N^T_{t_1}|$ free stubs and it therefore remains to show that $|P_i^T|$ and $|N_i^T|$ are for all $t_1<i=O(\log n)$ sufficiently concentrated around its expectation. Since exactly one stub at a vertex is selected in every round, the events are far from independent. To gain independence, we distribute the free stubs in $P_{i-1}$ in $\delta-1$ sets $P_{i-1,1}, \dots, P_{i-1,\delta-1}$ such that every set has either cardinality $\lfloor |P_{i-1}^t|/(\delta-1)\rfloor$ or $\lfloor |P_{i-1}^t|/(\delta-1)\rfloor+1$, and such that no two stubs of the same vertex are in the same set. Then, we apply the Chernoff bound (see Theorem~\ref{thm:chernoff}) to show that with probability $1-O(n^{-c})$ (where we can choose the $c$ arbitrarily large) for all those sets the number of stubs that create a new vertex and the number of stubs that remain free, are within a factor $(1\pm O(1/\log^2 n))$ of their expectation. The lemma follows by a union bound, since we have $3 (\delta-1)$ random variables, $O(\log n)$ rounds, and since
$$(1\pm O(1/\log^2n))^{O(\log n)}=(1\pm o(1)).$$
\soda{\qed}
\end{proof}
}We remark that Lemma~\ref{lem:tree_model} implies that
\begin{align*}
|N_i^T| &\geq (1\pm o(1))\frac{1}{\delta}|P_{i-1}^T|=(1\pm o(1)) \frac{1}{\delta}\left(2\left(1-\frac{1}{\delta}\right) \right)^{i-1-t_1} |P^T_{t_1}|\\
&\geq (1\pm o(1))\frac{(\delta-1)}{\delta}\left(2\left(1-\frac{1}{\delta} \right) \right)^{i-1-t_1}\log^5 n
\end{align*}
holds with probability $1-o(1/n)$ for all $t_1<i\le O(\log n)$. Recall that $c_{\mathcal{D}}= \frac{1}{\ln (2(1-\frac{1}{\delta}))}$, 
and note that the inequality
\begin{equation}
\label{eq:p2_time}
t_1+ c_{\mathcal{D}} (\ln n-5\ln\log n)\stackrel{\eqref{eq:cdondition_gamma}}{\le} t_1+c\ln n
\end{equation}
holds for $n$ large enough.
Therefore, for $i\geq t_1+c\ln n$, we have with probability $1-o(1/n)$ 
\begin{align}
|N_i^T|&\geq (1-o(1))\frac{\delta-1}{\delta}\left(2\left(1-\frac{1}{\delta}\right)\right)^{t_1+c_{\mathcal{D}}(\ln n-5\ln \log n)-1-t_1}\log^5n   \nonumber  \\
\label{eq:p2_newly_informed} 
&\geq (1-o(1)) \frac{(\delta-1)}{2\delta(1-\frac{1}{\delta})} n > \frac{(\delta-1)}{4\delta(1-\frac{1}{\delta})} n   \stackrel{\eqref{eq:cdondition_alpha}}{\geq} \alpha n.
\end{align}
\subsubsection{Coupling the two Processes}
\label{sec:coupling}
In order to apply the bound derived in \eqref{eq:p2_time} and \eqref{eq:p2_newly_informed} to $t_2$, we couple the DRP process to the tree process, in a way that $|\bar{N}_i|=|N^T_i|$ holds after every round $i\geq t_1$. 
We first explain the coupling and show then that we can run it long enough to inform at least $\alpha n$ vertices.

Recall that we consider the random graph process $M(\mathcal{D})$ with degree sequence $D_n$ and vertex set $V$,  where every vertex $v\in V$ has $d_v$ stubs attached. 
In round $i\geq t_1+1$, the DRP process selects at every informed vertex $v\in I_i$ one of the $d_v$ stubs uniformly at random. We use the same definition of free and used stubs as in the tree process. However, the fundamental difference to the tree processes is that a free stub $e'$ can become the target of a back match. In this case, $e'$ is matched without being able to increase the number of informed vertices. But since the degree of most of the vertices is larger than $\delta$, we have in every round more stubs that are selected for the first time than in the tree model. Therefore, we can assign to every target of a back match $e'$ another selected (and delayed stub) $twin(e')$ that can spread the rumor instead of $e'$, as soon as $e'$ is selected in for the first time. Moreover, we ignore all vertices of degree larger than $M$, as informing such vertices would increase the probability of back matches in later rounds. Whenever we match a free selected stub $e$ to a stub $e'\in_rU$, we say that $e'$ is \emph{good} if it is attached to an uninformed vertex of degree at most $M$. 

Recall that $\bar{N}_{i}$ denotes the set of non-sleeping vertices of degree at most $M$ that have been newly informed in round $i$.
We start the coupling with $\lceil\log^5 n\rceil$ vertices in $\bar{N}_{t_1}$ at time $t_1+1$. Note that we can do so by delaying the stubs at all but $\lceil\log^5 n\rceil$ vertices in $\bar{N}_{t_1}$ forever. Clearly, for $i=t_1$ it holds that $|\bar{N}_i|=|N^T_i|$, and our coupling will maintain this invariant. At time $i>t_1$ we select at every vertex $v\in \cup_{j=t_1}^{i-1}\bar{N}_{j}$ one stub u.a.r. and define $\bar{S}_i$ to be the set of all free stubs (or twins of free stubs) that have been selected. For an exact simulation of the DRP process, we have to match or delay all stubs in $\bar{S}_i$. We couple this decision with the random tree process, by first running round $i$ of the tree process and then simulating the DRP process, so that we inform exactly $|N_i^T|$ new vertices. 

Let $\bar{S}_i$ and $S^T_i$ be the sets of free stubs that are selected in round $i$ of the DRP process and the tree process. We first run Algorithm~\ref{alg:tree_algorithm} on $S^T_i$ and then Algorithm~\ref{alg:coupled_matching} on $\bar{S_i}$.  Algorithm~\ref{alg:coupled_matching} iteratively selects a stub $e$ from $\bar{S}_i$ and matches it to a randomly chosen stub $e'\in_R U$. If $e'$ is good, it informs the vertex at $e'$, otherwise it assigns a stub from $\bar{S}_i$ as twin to $e'$. The algorithm terminates as soon as the number of newly informed good vertices is exactly the number of new vertices in the tree model. 

\begin{algorithm}
\caption{Coupled DRP}
\label{alg:coupled_matching}
\begin{algorithmic}
\While{$|\bar{N}_i|<|N^T_i|$}
\State remove an arbitrary stub $e$ from $\bar{S}_i$.
\State choose $e'\in_R U$ and match $e$ to $e'$.
\If{$e'$ is good}
	\State inform the vertex at $e'$.
\Else
	\State remove an arbitrary stub $e_t$ from $\bar{S}_i$ and set $e_t$ sleeping.
	\State $twin(e'):=e_t$.
\EndIf
\EndWhile
\end{algorithmic}
\end{algorithm}

Clearly, Algorithm~\ref{alg:coupled_matching} assures that, as long as $\bar{S}_i$ is large enough, $|\bar{N}_i|$ is always equal to $|N^T_i|$, and we can therefore apply the bounds given by Lemma~\ref{lem:tree_model} to $|\bar{N}_i|$. The following lemma states that as long as we are in Phase 2, we always have enough free stubs in $\bar{S}_i$.
\begin{lemma}
\label{lem:coupling}
The condition $\bar{S}_i\ne \emptyset$ is with probability $1-o(1/n)$ never violated when running Algorithm~\ref{alg:coupled_matching} with $|I_i|<\alpha n$. 
\end{lemma}
Before we prove Lemma~\ref{lem:coupling}, we show how it implies this section's main lemma.  
\begin{proof}[Proof of Lemma~\ref{lem:phase_two}]
It is not very hard to see that Algorithm~\ref{alg:coupled_matching} correctly simulates the DRP process. The lemma therefore follows from Lemma~\ref{lem:coupling}, since $\alpha n\le  |N_{t_1+c\cdot \ln n}^T|$, as we already mentioned in \eqref{eq:p2_time}, holds with probability $1-o(1/n)$ and since $|N^T_i|=|\bar{N}_i|\le |I_i| $ holds for every $i\geq t_1+1$ for which $|I_i|\le \alpha n$.
\end{proof}

It remains to prove Lemma~\ref{lem:coupling}. 

\begin{proof}[Proof of Lemma~\ref{lem:coupling}]
We assume during the whole proof that \eqref{eq:stub_change} and \eqref{eq:newly_informed_change} hold for all considered integers $i$.  
We fix a round $t_1< i \le O(\log n)$ and assume that the statement was true for all previous rounds. Based on this assumption, we prove that the statement of the lemma holds with probability $1-o(1/n^2)$ for round $i$. The lemma then follows by induction (the statement is by definition true for $i=t_1$) and union bound.

Let  $\bar{N}_{i,k}:=\left\{v\in \bar{N}_i \middle| deg(v)=k\right\}$ and let $m_{i,k}:= |\bar{N}_{i,k}|/|\bar{N}_i|$. 
We claim that the expected number of selected stubs in round $i$ is
$$\sum_{t=t_1}^{i-1}|N_t|\left( \frac{\delta-1}{\delta} \right)^{i-t}
\quad \text{ in the tree model, and} \quad
\sum_{t=t_1}^{i-1}\sum_{k=\delta}^{M}|\bar{N}_{t}|\cdot m_{t,k}\left( \frac{k-1}{k} \right)^{i-t} $$

in the push process.
Indeed, let  $v$ be a vertex of degree $k$ that has been informed in round $t$. It is not very hard to see that the probability that one of the stubs at $v$ is selected for the first time in round $t+j$ is exactly $\left(\frac{k-1}{k}\right)^j$. 
Hence, the expectation of $|S_i|:=|\bar{S}_i|-|S_i^T|$ is at least 

\begin{align*}
&\sum_{t=t_1}^{i-1}\sum_{k=\delta}^M |\bar{N}_t|\cdot m_{t,k}\left[\left(\frac{k-1}{k}\right)^{i-t} -\left(\frac{\delta-1}{\delta}\right)^{i-t} \right]\geq \sum_{k=\delta}^M |\bar{N}_{i-1}|\cdot m_{i-1,k}\left[\frac{k-1}{k} -\frac{\delta-1}{\delta} \right]\\
&\geq \sum_{k=\delta+1}^M |\bar{N}_{i-1}|\cdot m_{i-1,k}\left[\frac{\delta}{\delta+1} -\frac{\delta-1}{\delta} \right]\geq \frac{1}{\delta^2+\delta} |\bar{N}_{i-1}|(1-m_{i-1, \delta}).
\end{align*}
If we consider the random variable $|S_i|$ in the setting of Theorem~\ref{lem:concentration}, then the coordinates are the outcome of the random experiments that select a stub at every vertex (in both processes) in every round. The effect of one coordinate $d_i$ is therefore at most one, and the number of coordinates that can contribute is at most 
$$2\left|\bigcup_{j=t_1}^{i-1}N^T_j \right| \cdot i = O\left(|N^T_{i-1}|\log n\right),$$
where the equality follows by Lemma~\ref{lem:tree_model} and the induction hypothesis.
Hence, by Theorem~\ref{lem:concentration} and since $|N^T_{i-1}|\geq \log^5n$, it follows that 
\begin{equation}
\label{eq:w.h.p_si}
\Pr\left[|S_i|\le \frac{1}{2(\delta^2+\delta)} |N^T_{i-1}|(1-m_{i-1, \delta})\right]=n^{-\Omega(\log^3 n)}=o(1/n^2)
\end{equation}
as long as $1-m_{i-1,\delta}=\Omega(1)$, which we will prove later. For now, we assume that $|S_i|\geq \frac{1}{2(\delta^2+\delta)} |N^T_{i-1}|(1-m_{i-1, \delta})$.

Note that $|S_i^T|$ is at least $|N_i^T|$ and that therefore 
$$|\bar{S_i}|\geq |N_i^T|+|S_i|\geq|N_i^T|+ \frac{1}{2(\delta^2+\delta)} |N^T_{i-1}|(1-m_{i-1, \delta}).$$ 
Whenever we match a new stub $e$ of $\bar{S}_i$ to a stub $e'\in_RU$, the probability that $e'$ is good is, by \eqref{eq:prob_not_informed_small}, at least $1-\gamma$. Suppose that we match $|N^T_{i}|+|S_i|/2$ times a stub of $\bar{S}_i$, and let $X$ count the number of times we match to a good stub. We will show that  $X$ is with probability $1-o(1/n^2)$ at least $|N_i^T|$, and thus $\bar{S}_i$ becomes never empty, since we match at most  $|N_i^T|+|S_i|/2-X\le |S_i|/2$ many times to stubs that are not good, so that the remaining $|S_i|/2$ stubs suffice as twins. 

It follows from Lemma~\ref{lem:tree_model} that 
\begin{align*}
|N^T_{i}|+|S_i|/2&\geq |N^T_{i}|+\frac{|N^T_{i-1}|(1-m_{i-1, \delta})}{4(\delta^2+\delta)}  \geq |N^T_i|\left(1+\frac{1-m_{i-1, \delta}}{8(\delta+1) (2\delta-2)} \right)\\
&\geq |N^T_i| \left(1+\frac{1-m_{i-1, \delta}}{16\delta^2} \right),
\end{align*}
and the expected number of matched stubs that are good is therefore at least
\begin{equation}
\mathbb{E}[X]  \stackrel{\eqref{eq:prob_not_informed_small}}{\geq} (1-\gamma)|N^T_i|  \left(1+\frac{1-m_{i-1, \delta}}{16\delta^2} \right).
\end{equation}
Note that each stub is matched to a good stub with probability at least $(1-\gamma)$ (independently of the other matches), and that $|N^T_i|=\Omega(\log^5n)$. In order to prove (by Theorem~\ref{thm:chernoff}) that $X\geq|N^T_i|$ holds with probability $1-n^{-\Omega(\log^4 n)}$, it therefore suffices to show that
\begin{equation}
\label{eq:chernoff_const}
(1-\gamma)\left( 1+(1-m_{i, \delta})\frac{1}{16\delta^2}\right) \geq1+\gamma.
\end{equation}
Recall that $m_{i-1, \delta}=|\bar{N}_{i-1, \delta}|/|\bar{N}_{i-1}|$ and observe that $|\bar{N}_{i-1, \delta}|$ is the sum of $|\bar{N}_{i-1}|$ indicator random variables, which are one with probability at most
\begin{align*}
\frac{D_{\delta} \delta}{(n-|I_{i-1}|-D_{\le \delta})(\delta+1)+D_{\delta}\delta} \le 1-\frac{n-|I_{i-1}|-D_{\le \delta}}{n-|I_{i-1}|}\le 1-\frac{1-\alpha-(1+o(1))\lambda_{\delta}}{1-\alpha},
\end{align*} 
where $D_{\delta}:=|\left\{j \mid  d_{j,n}=\delta\right\}|$ and $D_{\le \delta}:=|\left\{j \mid  d_{j,n}\le\delta\right\}|$.
As this bound holds for every i.r.v. deterministically, we conclude by Theorem~\ref{thm:chernoff} that 
$$\Pr\left[|\bar{N}_{i-1,\delta}|\geq \left(1-\frac{1-\alpha-\lambda_{\delta}}{2(1-\alpha)}|\bar{N}_{i-1}|\right)\right]\le n^{-\Omega(\log^4 n)}.$$
It follows that with probability $1-o(1/n^2)$
\begin{equation}
\label{eq:m_bound}
 m_{i-1, \delta}\le 1- \frac{1-\alpha -\lambda_{\delta}}{2(1-\alpha)} \stackrel{\eqref{eq:cdondition_alpha}}{=}\Omega(1),
\end{equation}
which settles \eqref{eq:w.h.p_si}.
Recalling $\alpha\le (1-\lambda_{\delta})/(1+\lambda_{\delta})$ from \eqref{eq:cdondition_alpha} we finish the proof by showing \eqref{eq:chernoff_const} as follows
\begin{align*}
(1-\gamma)\left( 1+(1-m_{i-1, \delta})\frac{1}{16\delta^2}\right) \stackrel{\eqref{eq:m_bound}}{\geq} (1-\gamma)\left( 1+\left(\frac{1}{2}-\frac{\lambda_{\delta}}{2(1-\alpha)}\right)\frac{1}{16\delta^2}\right) \\
\stackrel{\eqref{eq:cdondition_alpha}}{\geq} (1-\gamma)\left( 1+\frac{(1-\lambda_{\delta})}{64\delta^2}\right) \stackrel{\eqref{eq:cdondition_gamma}}{\geq} (1-\gamma)\left(1+\frac{2\gamma}{(1-\gamma)} \right)=1+\gamma. 
\end{align*}

\end{proof}
\full{\subsection{Phase 3}
\label{sec:phase3}
\soda{\vspace{-1mm}}
The following lemma states that for a fixed set of informed vertices $I_{t_1}\subset V$ of size $\lceil \alpha n \rceil$, there exists w.h.p for almost every vertex in $V\setminus I_{t_1}$ a short path that contains only vertices of small degree and ends at a vertex in $I_{t_1}$. 
\begin{lemma}
\label{lem:lem_short_path}
Let $t_2$ such that $|I_{t_2}|=\lceil \alpha n \rceil$ and let $\eps>0$. Then, there exist w.h.p for all but at most $o(n)$ vertices $v\in V\setminus I_{t_2}$ a path $v,v_1,  \dots, v_k$ such that 
\begin{enumerate}
\item $v_k\in I_{t_2}$
\item $k\in O(\log\log n)$
\item For all $1\le i \le k$ it holds that $deg(v_i) \le \log \log n$.
\end{enumerate}
If the degree sequence is bounded with constant maximum degree $\Delta$, then there exists with probability $1-o(1/n)$ for every vertex $v\in V\setminus I_{t_2}$ a path $v,v_1,  \dots, v_k$ with properties 1. and 2.
\end{lemma}

\begin{proof}
First, we claim that whenever $|I_i|=\alpha n$ there are at least $\alpha n$ unmatched stubs at informed vertices. We observe that there are in total at least $N:=\alpha n \delta $ stubs at informed vertices. If $F$ of those $N$ stubs are unmatched, then there are $(N-F)/2$ edges between vertices in $I_i$ and therefore $(N-F)/2-\alpha n+1$ of those edges where back matches. The probability of a back match is, by \eqref{eq:prob_not_informed_small}, at most $\gamma$ and it is a straightforward application of Theorem~\ref{thm:chernoff} to show that for $F<\alpha n$ the probability of this event is $e^{-\Omega(n)}$. We therefore assume that we have at least $\alpha n$ unmatched stubs at informed vertices. 
  
Let $v\in V\setminus I_{t_2}$ be an arbitrary vertex. We show that $v$ is w.h.p the starting vertex of a path that satisfies the properties of the lemma. Similarly as in the proof of Lemma~\ref{lem:phase_1} we describe a process that builds a 3-regular tree $T$ with root $v$, by matching first three stubs at $v$ and then recursively two new stubs at every neighbour. Note that we can always assume that all the vertices in the tree are uninformed, since otherwise the desired path is already established. Whenever we match to a vertex that is already in the tree, we let the corresponding branch of the tree \emph{die}. 

Let $|T_i|$ count the number of vertices on the $i$-th level of the tree. It follows from the sparseness of the degree sequence, in particular from
$$\sum_{j=0}^{\infty}\lambda_j j=O(1),$$
 that there exist a function $f(n)=o(1)$ such that 
$$\sum_{j=\log\log n}^n \lambda_j j=O(f(n)),$$
and therefore that the number of stubs at vertices of degree at least $\log\log n$ is 
$$\sum_{j=\log\log n}^nd_{j,n}j =O(n\cdot f(n)).$$
 As long as we do not hit a stub at an informed vertex, we have at least $|U|=\Omega(n)$ unmatched stubs and the probability that we match to a vertex of degree at least $\log \log n$ or a vertex that is already in the tree when we choose a stub u.a.r. from $U$ is therefore at most 
$$\chi_i= O\left( \frac{n\cdot f(n)+\sum_{j=1}^i|T_j|\log\log n}{n}\right),$$
which is $O(f(n)+ \log^2 n /n)=o(1)$ for $\sum_{j=1}^i|T_j|=O(\log n)$. 
Let
$$p_i:=\Pr\left[|T_i|< \left(\frac{6}{5}\right)^{i-1}\middle| |T_{i-1}|\geq \left(\frac{6}{5}\right)^{i-2}\right] \quad \text{ and let } \quad h:=\frac{\log\log n}{\log (6/5)}+1.$$
If we assume that $|T_{i-1}|\geq \left(\frac{6}{5}\right)^{i-2}$, then at least $4/5 (\frac{6}{5})^{i-2}$ of the $2(\frac{6}{5})^{i-2}$ stubs at level $i-1$ have to die in order to have $|T_i|<\left(\frac{6}{5}\right)^{i-2}$. This happens with probability at most
\begin{equation}\label{eq:pi}
p_i\le \binom{2\left(\frac{6}{5}\right)^{i-2}}{\frac{4}{5}\left(\frac{6}{5}\right)^{i-2}}\chi_i^{\frac{4}{5}\left(\frac{6}{5}\right)^{i-2}}\le 2^{2\left(\frac{6}{5}\right)^{i-2}}\chi_i^{\frac{4}{5}\left(\frac{6}{5}\right)^{i-2}} \le \left(4\chi_i^{4/5} \right)^{\left(\frac{6}{5}\right)^{i-2}}.
\end{equation}
Moreover, it is not hard to see that
$$\Pr\left[|T_2|<\left(\frac{6}{5}\right)\right] =\Pr[|T_2|<2]=O(\chi_2^2),$$
and the probability that $|T_{h}|< \log n$ is therefore at most
\begin{align*}
Pr\left[|T_2|<\frac{6}{5}\right]+\sum_{i=3}^{h} p_i &\le O(\chi_2^2) + \sum_{i=3}^{15} \left(4\chi_{15}^{4/5} \right)^{\left(\frac{6}{5}\right)^{i-2}} + \sum_{i=1}^{\infty} \left(4\chi_h^{4/5} \right)^i \\
&= O(\chi_{15}^{4/5})+ \frac{4\chi_h^{4/5}}{1-4\chi_h^{4/5}} =o(1).
\end{align*}
Thus, if we stop the tree process on level $h$ we have w.h.p at least $\log n$ free stubs. On the other hand, the vertices in $I_{t_2}$ have at least $\alpha n=\Omega(n)$ free stubs to which we can match. Hence, whenever we match one of the stubs at level $h$, we have a constant probability $\bar{p}$ to match to a stub of a vertex in $I_{t_2}$, and the probability that we do not match to at least one of those stubs is at most 
$$\bar{p}^{\Omega(\log n)}=n^{-\Omega(1)}.$$ 
We conclude that, as long as $|T|$ has enough vertices on level $h$ and we match at least one of them to $I_{t_1}$, vertex $v$ is adjacent to a path that satisfies the conditions of the lemma. Since this happens with probability at least $1-o(1)$, the expected number of vertices for which no such path exist is at most $o(n)$. It therefore follows from Markov's inequality that w.h.p for all but $o(n)$ vertices such a path exists. 

In order to prove the second statement of the lemma, we need to prove that the above calculations hold with probability $1-o(1/n^2)$ for bounded degree sequences. Let
$$h:=\frac{\log\log n^3 - \log\log (\frac{\alpha}{\Delta})}{\log(6/5)} +1$$
be the height of the tree that we build.  
For a constant maximum degree $\Delta$, the probability $\chi_i$ is at most
$$\chi_i= O\left( \frac{2^i \Delta}{n\delta}\right)=O\left(\frac{2^i}{n}\right)$$
for $i\le h$ and $|V\setminus I_{t_1}|=\Omega(n)$.
It is therefore not very hard to see that 
$$\Pr\left[|T_7|<\left(\frac{6}{5}\right)^6 \right]= \Pr[|T_7|<3] =O(1/n^3)$$
and it follows by \eqref{eq:pi} that for $8\le i \le h$
$$p_i = O\left(\frac{\log n}{n}\right)^{\frac{4}{5}(\frac{6}{5})^6}=o(n^{-2.3}).$$
Thus, the probability that we have at least $(\frac{6}{5})^{h-1}=3\log n / \log(\alpha/\Delta)$ free stubs on level $h$ is at least
$$1-\left(\Pr\left[|T_7|<\left(\frac{6}{5}\right)^6 \right]+\sum_{i=8}^h p_i \right)=1-o(1/n^2).$$
Finally, since the probability that we match to an informed vertex when we match the stubs at $T_h$ is at least $\bar{p}:=\frac{\alpha n}{\Delta n}=\frac{\alpha}{\Delta}$ for every stub, the probability that we match none of the stubs to an informed vertex is at most $\bar{p}^{3\log n / \log(\alpha/\Delta)}=O(1/n^3)$, and the second statement of the lemma therefore follows by a first moment argument. 
\end{proof}

\soda{In this extended abstract we omit the proof of Lemma~\ref{lem:lem_short_path}, but we show that such a good path transmits the rumor fast, which implies Lemma~\ref{lem:phase_3}.} 
\full{With Lemma~\ref{lem:lem_short_path} at hand, we can prove Lemma~\ref{lem:phase_3} by exploiting only properties of the random graph. }
\begin{proof}[Lemma~\ref{lem:phase_3}]
Note that Lemma~\ref{lem:lem_short_path} is stated for $|I_{t_2}|=\lceil \alpha n \rceil$. This is not a problem since we can switch from phase 2 to phase 3 as soon as $\lceil \alpha n \rceil$ vertices are informed. 

We first handle unbounded degree sequences. By Lemma~\ref{lem:lem_short_path}, we have w.h.p for all but $o(n)$ vertices a good path. We claim that the rumor spreads over such a path with probability at least $1-O(1/\log^2 n)$ in $o(\log n)$ rounds. Let $t:=3(\ln\log n)\cdot(\log\log n)$. Then for all $1\le j \le k$, the probability that it takes more than $t$ rounds until $v_j$ pushes the rumor to $v_{j-1}$ is at most
$$\left(1-\frac{1}{\log\log n} \right)^t\le e^{-\frac{t}{\log\log n}} =\frac{1}{\log^3n},$$
and the probability that there is at least one vertex on the path that needs more than $t$ rounds to push the rumor is at most $O(1/\log^2 n)$. Therefore, the expected number of vertices for which such a path exist but it does not transfer the rumor in time $k\cdot t=O((\log\log n)^3)$ is $o(n)$ and it follows by Markov's inequality that for every $\eps>0$ with high probability at least $(1-\eps)n$ vertices are informed at time $t_2+k\cdot t=t_2+O((\log\log n)^3)$.

For bounded degree sequences, we have that with probability $1-o(1/n)$ every vertex $v$ has path as described in Lemma~\ref{lem:lem_short_path}. We say that the rumor is delayed if it does not travel one step further on the path during a round of the protocol. Suppose that the rumor needs more than $t'$ steps to travel from $v_k$ to $v$. This means that the rumor was delayed at least $t'-k$ times. Let $t':=6\Delta^2\ln n$. Since the expected number of delays in $t'$ rounds is at most $t'(\Delta-1)/\Delta$, it follows by the Chernoff bound that the rumor reaches $v_k$ with probability $1-o(1/n^2)$ after $t'$ rounds and we can finish the proof with a union bound. 
\end{proof}
}

\vspace{-3mm}
\bibliographystyle{amsplain}
\bibliography{rumor}
\end{document}